\documentclass[12pt,reqno]{amsart}

\usepackage{amsthm,amsfonts,amscd,amsmath}
\usepackage{graphicx}
\usepackage[top=20mm, bottom=20mm, left=30mm, right=30mm]{geometry}
\usepackage[colorlinks=true,citecolor=black,linkcolor=black,urlcolor=magenta]{hyperref}

\usepackage{fancyvrb}
\usepackage{ulem} 
\usepackage[T1]{fontenc} 
\usepackage{pbsi} 
\usepackage{url} 
\usepackage{calc} 
\usepackage{datetime}

\usepackage{dashrule}
\def\dashfill{\cleaders\hbox to 2em{-}\hfill}

\renewcommand{\emph}[1]{\textit{#1}}
\newcommand{\citep}{\cite}

\usepackage{enumitem}
\setlist[itemize]{leftmargin=0.35in}

\usepackage{diagbox}

\newcommand{\QPochhammer}[3]{\left(#1; #2\right)_{#3}} 

\newcommand{\cf}{\textit{cf.\ }} 
\newcommand{\Iverson}[1]{\ensuremath{\left[#1\right]_{\delta}}} 

\DeclareMathOperator{\poly}{poly}
\DeclareMathOperator{\Num}{Num}
\DeclareMathOperator{\Denom}{Denom}

\title[Maxie D. Schmidt $\qquad$ \texttt{maxieds@gmail.com}]{
       New Recurrence Relations and Matrix Equations for Arithmetic 
       Functions Generated by Lambert Series} 

\author{Maxie D. Schmidt \\ 
        \href{mailto:maxieds@gmail.com}{maxieds@gmail.com}}         
\address{School of Mathematics \\ 
         Georgia Institute of Technology \\
         Atlanta, GA 30332 \\ 
         USA
        }
\email{maxieds@gmail.com}
\date{\today}


\allowdisplaybreaks 

\theoremstyle{plain} 
\newtheorem{theorem}{Theorem}

\newtheorem{lemma}[theorem]{Lemma}
\newtheorem{cor}[theorem]{Corollary}
\numberwithin{theorem}{section}

\theoremstyle{definition} 
\newtheorem{definition}[theorem]{Definition}

\newtheorem{remark}[theorem]{Remark}

\begin{document}

\begin{abstract} 
We consider relations between the pairs of sequences, $(f, g_f)$, generated by the 
Lambert series expansions, $L_f(q) = \sum_{n \geq 1} f(n) q^n / (1-q^n)$, in $q$. 
In particular, we prove new forms of 
recurrence relations and matrix equations defining these sequences for all 
$n \in \mathbb{Z}^{+}$. The key ingredient to the proof of these results is 
given by the statement of Euler's pentagonal number theorem expanding the 
series for the infinite $q$-Pochhammer product, 
$(q; q)_{\infty}$, and for the first $n$ terms of the partial products, 
$(q; q)_n$, forming the denominators of the rational $n^{th}$ partial sums of 
$L_f(q)$. Examples of the new results given in the article include 
new exact formulas for and applications to 
the Euler phi function, $\phi(n)$, the M\"obius function, $\mu(n)$, the 
sum of divisors functions, $\sigma_1(n)$ and $\sigma_{\alpha}(n)$, for 
$\alpha \geq 0$, and to Liouville's lambda function, $\lambda(n)$. 
\end{abstract} 

\maketitle

\section{Introduction} 

\subsection{Overview and motivation}

Our new results provide \emph{exact} matrix-based formulas 
for a wide range of classical special 
arithmetic functions expanded in well-known Lambert series expansions of the 
form defined in the next subsection. The first form of the 
exact formulas for these special case arithmetic functions, $f(n)$ 
are stated through a matrix factorization result of the following form for all 
natural numbers $n \geq 0$: 
\begin{align} 
\label{eqn_first_matrix_factorization_result_stmt_v1} 
\left(f(k)\right)_{1 \leq k \leq n} & = A_n^{-1} \cdot 
     \left(B_m(f)\right)_{0 \leq m < n}. 
\end{align} 
The $n \times n$ matrices $A_n$ and $A_n^{-1}$ implicit to the last equation are 
always independent of the choice of the function $f(n)$. Moreover, the right-hand-side 
vector whose entries are given by $B_m(f)$ is independent of the $A_n$ and depends 
only on a finite sum determined by $f$ for all $m,n \geq 0$. 

The matrix equation in \eqref{eqn_first_matrix_factorization_result_stmt_v1} 
reflects a more general class of so-termed \emph{Lambert series factorization} 
results of the form 
\[
\sum_{n \geq 1} \frac{f(n) q^n}{1-q^n} = \frac{1}{C(q)} \sum_{n \geq 0} 
     \sum_{k=1}^n s_{n,k} f(k) \cdot q^n, 
\]
where the invertible, lower triangular matrix $A_n$ in the form of 
\eqref{eqn_first_matrix_factorization_result_stmt_v1} corresponds to the 
entries $s_{n,k}$, which are also independent of the function $f$, 
for fixed $n \geq 1$. In the cases of these more general factorizations 
presented in this article, which we derive and prove by separate means here, 
we always have that the series expansion of the factorization 
parameter $C(q)$ is defined through the \emph{$q$-Pochhammer symbol} as 
$C(q) \equiv \QPochhammer{q}{q}{\infty}$. 

Examples of the new formulas for the special arithmetic functions, 
$\mu(n)$, $\phi(n)$, and $\lambda(n)$, that we are able to obtain through the 
forms of the mew matrix factorization equations we prove within the article 
include 
\begin{align*} 
\mu(n) & = \sum_{k=1}^n \left(\sum_{d|n} p(d-k) \mu(n/d)\right) \cdot B_{k-1}(\mu) \\ 
\phi(n) & = \sum_{k=1}^n \left(\sum_{d|n} p(d-k) \mu(n/d)\right) \cdot B_{k-1}(\phi) \\ 
\lambda(n) & = \sum_{k=1}^n \left(\sum_{d|n} p(d-k) \mu(n/d)\right) \cdot B_{k-1}(\lambda), 
\end{align*} 
where $p(n)$ denotes \emph{Euler's partition function} and where the special vector 
entries, $B_m(f)$, from \eqref{eqn_first_matrix_factorization_result_stmt_v1} 
are defined as follows for each $m \geq 0$\footnote{ 
     \underline{\emph{Notation}}: 
     \emph{Iverson's convention} compactly specifies 
     boolean-valued conditions and is equivalent to the 
     \emph{Kronecker delta function}, $\delta_{i,j}$, as 
     $\Iverson{n = k} \equiv \delta_{n,k}$. 
     Similarly, $\Iverson{\mathtt{cond = True}} \equiv 
                 \delta_{\mathtt{cond}, \mathtt{True}}$ 
     in the remainder of the article. 
}: 
\begin{align*} 
B_m(\mu) & := \Iverson{m = 0} + \sum_{b = \pm 1} 
     \sum_{k=1}^{\lfloor \frac{\sqrt{24m+1}-b}{6} \rfloor} 
     (-1)^k \Iverson{m+1-k(3k+b)/2 = 1} \\ 
B_m(\phi) & := m+1 - 
     \sum_{b = \pm 1} 
     \sum_{k=1}^{\lfloor \frac{\sqrt{24m+1}-b}{6} \rfloor} 
     (-1)^{k+1} \left(m+1-\frac{k(3k+b)}{2}\right) \\ 
B_m(\lambda) & := \Iverson{\sqrt{m+1} \in \mathbb{Z}} - 
     \sum_{b = \pm 1} 
     \sum_{k=1}^{\lfloor \frac{\sqrt{24m+1}-b}{6} \rfloor} 
     (-1)^{k+1} \Iverson{\sqrt{m+1-k(3k+b)/2} \in \mathbb{Z}}. 
\end{align*} 
The results we prove within the article also include new recurrence relations 
for the computations of the average order of special arithmetic functions, 
denoted by $\sigma_{g_f, x} := \sum_{n \leq x} g_f(n)$, for fixed functions 
$f(n)$ and its corresponding $g_f(n)$ when $x \geq 1$. The next subsections 
make the expansions of the Lambert series expansions we consider and the 
main theorems proved within the article precise. 

\subsection{Lambert series generating functions} 

In this article, we consider new recurrence relations and matrix equations 
related to \emph{Lambert series} expansions of the form 
\citep[\S 27.7]{NISTHB} \citep[\S 17.10]{HARDYANDWRIGHT} 
\begin{align}
\label{eqn_LambertSeriesfb_def} 
L_f(q) & := 
\sum_{n \geq 1} \frac{f(n) q^n}{1-q^n} = \sum_{m \geq 1} g_f(m) q^m,\ |q| < 1, 
\end{align} 
for prescribed functions $f: \mathbb{Z}^{+} \rightarrow \mathbb{C}$, and some 
$g_f: \mathbb{Z}^{+} \rightarrow \mathbb{C}$ where $g_f(m) = \sum_{d | m} f(d)$. 
There are many well-known Lambert series for special arithmetic functions of the 
form in \eqref{eqn_LambertSeriesfb_def}. 
Examples include the following series where $\mu(n)$ denotes the 
\emph{M\"obius function}, $\phi(n)$ denotes \emph{Euler's totient function}, 
$\sigma_{\alpha}(n)$ denotes the generalized \emph{sum of divisors function}, and 
$\lambda(n)$ denotes \emph{Liouville's function} \citep[\S 27.7]{NISTHB}: 
\begin{align} 
\label{eqn_WellKnown_LamberSeries_Examples} 
\sum_{n \geq 1} \frac{\mu(n) q^n}{1-q^n} & = q, && 
     (f, g_f) := (\mu(n), \Iverson{n = 1}) \\ 
\notag
\sum_{n \geq 1} \frac{\phi(n) q^n}{1-q^n} & = \frac{q}{(1-q)^2}, && 
     (f, g_f) := (\phi(n), n) \\ 
\notag
\sum_{n \geq 1} \frac{n^{\alpha} q^n}{1-q^n} & =  
     \sum_{m \geq 1} \sigma_{\alpha}(n) q^n, && 
     (f, g_f) := (n^{\alpha}, \sigma_{\alpha}(n)) \\ 
\notag
\sum_{n \geq 1} \frac{\lambda(n) q^n}{1-q^n} & = \sum_{m \geq 1} q^{m^2}, && 
     (f, g_f) := (\lambda(n), \Iverson{\text{$n$ is a positive square}}). 
\end{align}

\subsection{New Results} 
\label{Section_NewResults}

We have two interesting cases of \eqref{eqn_LambertSeriesfb_def} to consider: 
\begin{itemize} 

\item[\textbf{1. }] 
           The case where $f(n)$ is our arithmetic function of interest that 
           we wish to study, i.e., in the first, second, and fourth equations in 
           \eqref{eqn_WellKnown_LamberSeries_Examples}; and 
\item[\textbf{2. }] 
            The case where $g_f(n)$ is the interesting arithmetic function we 
            wish to study, i.e., in the third equation from 
            \eqref{eqn_WellKnown_LamberSeries_Examples}. 

\end{itemize} 

\subsubsection{Case 1} 

In the first case, for each $n \geq 1$ and $i \leq n$ we are able to form the 
matrix solutions for $f(n)$ given in 
Theorem \ref{prop_MatrixEquations_fi_eq_AinvBbm}, which 
are expanded in terms of the sequences in the next definition. 

\begin{definition} 
\label{def_afn_ani_matrix_seqs}
For integers $n \geq 0$, the sequences, $a_f(n)$ and $a_{n,i}$, are 
defined to be (\cf Remark \ref{remark_authors_note_on_AnAnInvEntryForms} 
on page \pageref{remark_authors_note_on_AnAnInvEntryForms})\footnote{ 
     \label{footnote_expl_of_floored_terms}
     The floored terms, $\left\lfloor (\sqrt{24n+1}-b) / 6 \right\rfloor$, for 
     $b = \pm 1$ in this and in subsequent formulas in the article 
     correspond to solving for the upper bounds on $k$ in the following 
     inequalities (\cf footnote \ref{footnote_pentagonal_number_sets} on 
     page \pageref{footnote_pentagonal_number_sets}): 
     \begin{align*} 
     0 \leq \frac{k(3k+b)}{2} \leq n & \iff 
     0 \leq k \leq \frac{\sqrt{24n+1}-b}{6}. 
     \end{align*} 
} 
\begin{align*} 
a_f(n) & = 
\sum_{i=1}^{n} f(i) \left( 
     \underset{:= a_{n,i}}{\underbrace{ 
     \sum_{\substack{(k, s) : (s+1)i+k(3k\pm 1) / 2 = n \\ k, s \geq 0}} (-1)^k}}\right) + 
     \Iverson{n = 0} \\ 
a_{n,i} & = 
     \sum_{b = \pm 1} \sum_{s=0}^{\lfloor \frac{n}{i} \rfloor - 1} 
     (-1)^{\lfloor \frac{\sqrt{24(n-(s+1)i)+1}-b}{6} \rfloor} \cdot 
     \Iverson{\frac{\sqrt{24(n-(s+1)i)+1}-b}{6} \in \mathbb{Z}}. 
\end{align*} 
For $n \geq 1$, we define the $n \times n$ matrices, $A_n$ and $A_n^{-1}$, 
in terms of these sequences as follows: 
\begin{equation} 
\label{eqn_AnAnInv_MatixDefs} 
A_n := (a_{i,j})_{1 \leq i,j \leq n},\ A_n^{-1} := \left(a_{i,j}^{(-1)}\right)_{1 \leq i,j \leq n}. 
\end{equation} 
The matrices, $A_n$ and $A_n^{-1}$, are independent of the choice of the 
function $f$ for all $n$ and are each invertible, lower triangular square matrices with 
ones on their diagonals. The independence of these matrices on the choice of $f$ 
implicit to the expansions in \eqref{eqn_LambertSeriesfb_def} 
leads to the Lambert series matrix factorization result phrased by 
Theorem \ref{prop_MatrixEquations_fi_eq_AinvBbm} below. 
\end{definition} 

\noindent 
The motivation for defining the sequences, $a_f(n)$ and $a_{n,i}$, in 
Definition \ref{def_afn_ani_matrix_seqs} is to provide a compact notation for 
expressing the left-hand-side terms, $A_n [f(1)\ \cdots\ f(n)]^{T}$, of the 
matrix equation corresponding to the non-homogeneous recurrence relations for 
$g_f(n)$ given in Theorem \ref{prop_MatrixEquations_fi_eq_AinvBbm} and in 
Theorem \ref{prop_bn_recs} 
(see also Lemma \ref{lemma_LS_partial_sum_exps} on page 
\pageref{lemma_LS_partial_sum_exps}). 
Particular examples of the sequences, $a_f(n)$, include the next 
special case for the generalized sum of divisors functions, 
$f(n) := \sigma_{\alpha}(n) = \sum_{d|n} d^{\alpha}$, 
for fixed $\alpha \in \mathbb{C}$ in 
\eqref{eqn_afn_special_cases}. The expansion in \eqref{eqn_afn_special_cases} below 
provides a listing of the terms, $a_f(n)$, in the previous definition corresponding to 
the Lambert series pair $(f(n), g_f(n)) := (n^{\alpha}, \sigma_{\alpha}(n))$ in 
\eqref{eqn_LambertSeriesfb_def} which is intended for comparison with the 
closely-related matrix factorization result for this special case given by 
Theorem \ref{prop_MatrixEquations_fi_eq_AinvBbm}. 
\begin{align} 
\label{eqn_afn_special_cases} 
\sum_{m \geq 0} a_{n^\alpha}(m) q^m & = 1 + q + 2^{\alpha} q^2 + 
     \left(-1-2^{\alpha}+3^{\alpha}\right)q^3 + 
     \left(-1-3^{\alpha}+4^{\alpha}\right)q^4 \\ 
\notag 
     & \phantom{=q\ } + 
     \left(-1-2^{\alpha}-3^{\alpha}-4^{\alpha}+5^{\alpha}\right)q^5 + 
     \left(3^{\alpha}-4^{\alpha}-5^{\alpha}+6^{\alpha}\right) q^6 \\ 
\notag 
     & \phantom{=q\ } + 
     \left(-3^{\alpha}-5^{\alpha}-6^{\alpha}+7^{\alpha}\right) q^7 + \cdots. 
\end{align} 
The first few cases of the matrices, $A_n \in \mathbb{Z}_{n\times n}$, and their 
corresponding inverses, $A_n^{-1} \in \mathbb{N}_{n\times n}$, are shown in 
Table \ref{table_matrices_An}. 
In general, we see that for all $n \geq 2$, we have that 
\begin{align*} 
A_{n+1}^{-1} & = 
     \left[
     \begin{array}{c|c} 
     A_n^{-1} & \mathbf{0} \\ \hline %\dashfill %\hdashline 
     r_{n+1,n}, \ldots, r_{n+1, 1} & 1 
     \end{array}
     \right], 
\end{align*} 
where the first several special cases of the sequences, 
$\{r_{n+1,n}, r_{n+1,n-1}, \ldots, r_{n+1,1}\}$, are given in 
Table \ref{table_matrix_An_last_row_seqs}. 
The statement of the next theorem employs these sequences and 
matrix forms. 
The proof of Theorem \ref{prop_MatrixEquations_fi_eq_AinvBbm} 
is given in Section \ref{Section_Proofs} below. 

\begin{remark}[Short Author's Note] 
\label{remark_authors_note_on_AnAnInvEntryForms} 
Since the first submission of the manuscript 
the entries, $a_{n,i}$ and $a_{n,i}^{(-1)}$, in \eqref{eqn_AnAnInv_MatixDefs} 
have been determined in closed-form through joint work by Merca and Schmidt on 
Lambert series factorizations (2017). 
In particular, we have that $a_{n,i}$ corresponds to\footnote{ 
     \underline{Notation:} The bracket notation for coefficient extraction of a 
     (formal) power series is defined to be $[q^n] F(q) := f_n$ when 
     $F(q) := \sum_n f_n q^n$ represents the ordinary generating function of the 
     sequence, $\langle f_n \rangle_{n \geq 0}$. This notation is also employed in 
     Section \ref{Section_Proofs} below. 
} 
\[
s_o(n, i) - s_e(n, i) = [q^n] \frac{q^i}{1-q^i} (q; q)_{\infty}, 
\]
where $s_o(n, k)$ and $s_e(n, k)$ are respectively the number of $k$'s in all 
partitions of $n$ into an odd (even) number of distinct parts. 
Similarly, we can derive an exact formula for the inverse matrix entries as 
\[
a_{n,i}^{(-1)} = \sum_{d|n} p(d-i) \mu(n/d), 
\]
where $p(n)$ denotes \emph{Euler's partition function} which is generated by the 
reciprocal of the $q$-Pochhammer symbol as 
$p(n) = [q^n] \QPochhammer{q}{q}{\infty}^{-1}$ for all $n \geq 0$. 
\end{remark} 

\begin{table}[ht!] 
\centering 
\begin{tabular}{||c||c|c||} \hline\hline
$n$ & $A_n$ & $A_n^{-1}$ \\ \hline 
1 & $[1]$ & $[1]$ \\ 
2 & $\begin{bmatrix} 1 & 0 \\ 0 & 1 \end{bmatrix}$ & 
    $\begin{bmatrix} 1 & 0 \\ 0 & 1 \end{bmatrix}$ \\ 
3 & $\begin{bmatrix} 1 & 0 & 0 \\ 0 & 1 & 0 \\ -1 & -1 & 1 \end{bmatrix}$ & 
    $\begin{bmatrix} 1 & 0 & 0 \\ 0 & 1 & 0 \\ 1 & 1 & 1 \end{bmatrix}$ \\ 
4 & $\begin{bmatrix} 1 & 0 & 0 & 0 \\ 0 & 1 & 0 & 0 \\ -1 & -1 & 1 & 0 \\ 
                     -1 & 0 & -1 & 1 \end{bmatrix}$ & 
    $\begin{bmatrix} 1 & 0 & 0 & 0 \\ 0 & 1 & 0 & 0 \\ 1 & 1 & 1 & 0 \\ 
                     2 & 1 & 1 & 1 \end{bmatrix}$ \\ 
5 & $\begin{bmatrix} 1 & 0 & 0 & 0 & 0 \\ 0 & 1 & 0 & 0 & 0 \\ 
                     -1 & -1 & 1 & 0 & 0 \\ -1 & 0 & -1 & 1 & 0 \\ 
                     -1 & -1 & -1 & -1 & 1 \end{bmatrix}$ & 
    $\begin{bmatrix} 1 & 0 & 0 & 0 & 0 \\ 0 & 1 & 0 & 0 & 0 \\ 
                     1 & 1 & 1 & 0 & 0 \\ 2 & 1 & 1 & 1 & 0 \\ 
                     4 & 3 & 2 & 1 & 1 \end{bmatrix}$ \\ 
\hline\hline
\end{tabular} 

\bigskip
\caption{The first few special cases of the matrices, $A_n$ and $A_n^{-1}$} 
\label{table_matrices_An} 

\end{table} 

\begin{table}[ht] 
\centering 
\begin{tabular}{||c||l||} \hline\hline
$n$ & $\{r_{n,n-1}, r_{n,n-2}, \ldots, r_{n,1}\}$ \\ \hline 
2 & $\{1\}$ \\ 
3 & $\{1, 1\}$ \\ 
4 & $\{2, 1, 1\}$ \\ 
5 & $\{4, 3, 2, 1\}$ \\ 
6 & $\{5, 3, 2, 2, 1\}$ \\ 
7 & $\{10, 7, 5, 3, 2, 1\}$ \\ 
8 & $\{12, 9, 6, 4, 3, 2, 1\}$ \\ 
9 & $\{20, 14, 10, 7, 5, 3, 2, 1\}$ \\ 
10 & $\{25, 18, 13, 10, 6, 5, 3, 2, 1\}$ \\ 
11 & $\{41, 30, 22, 15, 11, 7, 5, 3, 2, 1\}$ \\ 
12 & $\{47, 36, 26, 19, 14, 10, 7, 5, 3, 2, 1\}$ \\ 
\hline\hline
\end{tabular} 

\bigskip 
\caption{The bottom row sequences in the matrices, $A_n^{-1}$} 
\label{table_matrix_An_last_row_seqs} 

\end{table} 

\begin{theorem}[Matrix Factorization Equations for $f(n)$] 
\label{prop_MatrixEquations_fi_eq_AinvBbm} 
For all $n \geq 1$, we have the following matrix factorization equations exactly 
generating the arithmetic functions, $f(n)$, in the definition of 
\eqref{eqn_LambertSeriesfb_def}: 
\begin{align} 
\label{eqn_fn_matrix_eqn}
\begin{bmatrix} f(1) \\ f(2) \\ \vdots \\ f(n) \end{bmatrix} & = 
     A_n^{-1} \left( 
     \underset{:= (B_{g_f,m})}{\underbrace{
     g_f(m+1) - \sum_{b = \pm 1} \sum_{k=1}^{\lfloor \frac{\sqrt{24m+1}-b}{6} \rfloor} 
     (-1)^{k+1} g_f(m+1-k(3k+b)/2)}} 
     \right)_{0 \leq m < n}. 
\end{align} 
\end{theorem} 

\subsubsection{Case 2} 

In the second case, we have recurrence relations for $g_f(n)$ in 
\eqref{eqn_LambertSeriesfb_def}, of the form stated in 
Theorem \ref{prop_bn_recs}. 
Moreover, if we define $\Sigma_{g_f,x} := \sum_{n \leq x} g_f(n)$ to be the 
average order of $g_f(n)$, then we can also 
prove easily by induction that $\Sigma_{g_f,x}$ itself also satisfies the 
related form of the recurrence relation given in 
Corollary \ref{cor_Sigmabx_recs}. 

\begin{theorem}[Recurrence Relations for $g_f(n)$] 
\label{prop_bn_recs} 
For all $n \geq 1$, we have the following recurrence relation for $g_f(n)$ 
expanded in terms of the sequences from 
Definition \ref{def_afn_ani_matrix_seqs}: 
\begin{align*} 
g_f(n+1) & = \sum_{b = \pm 1} \sum_{k=1}^{\lfloor \frac{\sqrt{24n+1}-b}{6} \rfloor} 
     (-1)^{k+1} g_f(n+1-k(3k+b) / 2) + a_f(n+1). 
\end{align*} 
\end{theorem} 

\begin{cor}[Recurrence Relations for $\Sigma_{g_f,x}$] 
\label{cor_Sigmabx_recs} 
Let the $x^{th}$ partial sums of the function, $g_f(n)$, i.e., its average order, 
be defined by $\Sigma_{g_f,x} := \sum_{n \leq x} g_f(n)$. 
Then for all $n \geq 1$, we have that 
\begin{align} 
\label{eqn_Sigmabx_recs}
\Sigma_{g_f,n+1} & = \sum_{b = \pm 1} \left( 
     \sum_{k=1}^{\lfloor \frac{\sqrt{24n+1}-b}{6} \rfloor + 1} 
     (-1)^{k+1} \Sigma_{g_f,n+1-k(3k+b)/2} + 
     \sum_{k=1}^{n} a_f(k+1)\right). 
\end{align} 
\end{cor} 

\noindent 
Each of Theorem \ref{prop_bn_recs} and Corollary \ref{cor_Sigmabx_recs} 
are proved in Section \ref{Section_Proofs}. 

\subsubsection{Algorithms for computing the functions, $f(n)$ and $g_f(n)$, in polynomial time} 

Since the determinant of a $(n-1) \times (n-1)$ matrix can be computed in 
$O(n^3)$ time, if $g_f(n)$ can be computed in constant time, then by the 
theorem, we have a $O(n^4)$ polynomial time algorithm to compute any 
function $f(n)$ in \eqref{eqn_LambertSeriesfb_def}. 
If we instead use Gaussian elimination with back substitution, we can 
compute the functions, $f(n)$, in $O(n^3)$ time. 
Similarly, if $g_f(n)$ can be computed in $O(h_{g_f}(n))$ time, 
then we can compute any 
function $f(n)$ in \eqref{eqn_LambertSeriesfb_def} in 
$O(h_{g_f}(n) \sqrt{n} + n^4)$ / $O(h_{g_f}(n) \sqrt{n} + n^3)$ time. 

However, we note that each of the special arithmetic functions on the 
left-hand-side of \eqref{eqn_WellKnown_LamberSeries_Examples} can be 
computed more efficiently using sieves and other prime factorization 
algorithms. 
Despite more efficient known methods for computing the classical arithmetic 
functions involved in the expansions of \eqref{eqn_WellKnown_LamberSeries_Examples}, 
this observation is still useful since it 
implies that there are now known polynomial-time algorithms for computing the 
pairs, $(f(n), g_f(n))$, in \emph{any} Lambert series expansion provided a 
polynomial-time method of computation for either one of the 
functions in the corresponding pair. 

\subsection{Organization of the article} 

The proofs of Theorem \ref{prop_MatrixEquations_fi_eq_AinvBbm}, 
Theorem \ref{prop_bn_recs}, and of Corollary \ref{cor_Sigmabx_recs} 
stated in the last subsection are given first in Section \ref{Section_Proofs}. 
In Section \ref{Section_Examples}, we provide several concrete examples of the 
applications of these new results to the classical arithmetic functions of the 
sum of divisors function, $\sigma_1(n)$, the Euler phi function, $\phi(n)$, the 
M\"obius function, $\mu(n)$, and to the Liouville lambda function, $\lambda(n)$. 
In the concluding remarks we give in Section \ref{Section_Concl}, we suggest an 
approach to analogs of the new results proved within the article for 
a generalized class of Lambert series expansions, 
$L_{f}(\alpha, \beta; a, b, c, d; q)$, 
which are suggested as a future avenue of research on this topic. 

\section{Proofs of the theorems} 
\label{Section_Proofs} 

In order to obtain recurrence relations between the sequences implicit to the 
definition of \eqref{eqn_LambertSeriesfb_def}, we first observe that for all 
$m \geq 0$ we have the next series expansions of the partial sums of the 
Lambert series, $L_f(q)$, 
where $\QPochhammer{q}{q}{n} = (1-q)(1-q^2) \cdots (1-q^{n})$ denotes the 
\emph{$q$-Pochhammer symbol} \citep[\S 17.2]{NISTHB}, and where the functions 
$\poly_{i,m}(f; q)$ denote polynomials in $q$ with coefficients depending on $f$ 
for $i = 1,2$ and whose degree is linear in the fixed index $m$. 

\begin{lemma}[Partial Sums of the Lambert Series, $L_f(q)$] 
\label{lemma_LS_partial_sum_exps} 
For a fixed pair of functions $(f(n), g_f(n))$ in the expansions of 
\eqref{eqn_LambertSeriesfb_def} and for all integers $m \geq 0$ we 
have that 
\begin{subequations} 
\begin{align} 
\label{eqn_bmp1_coeff_exp_va} 
g_f(m+1) & = [q^m]\left(\frac{1}{q} \times \sum_{n=1}^{m+1} 
     \frac{f(n) q^n}{1-q^n}\right) \\ 
\label{eqn_bmp1_coeff_exp_vb} 
     & = 
     [q^m]\left(\frac{\frac{1}{q} \cdot \QPochhammer{q}{q}{m+1} 
     \left[ \frac{f(1) \cdot q}{1-q} + 
     \frac{f(2) \cdot q^2}{1-q^2} + \cdots + \frac{f(n) \cdot q^{m+1}}{1-q^{m+1}}\right]}{ 
     (1-q)(1-q^2) \cdots (1-q^{m+1})}\right) \\ 
\label{eqn_bmp1_coeff_exp_vc} 
     & = 
     [q^m]\left(\frac{\sum_{1 \leq i \leq m+1} a_f(i) q^i + 
     q^{m+2} \cdot \poly_{1,m}(f; q)}{ 
     1 + \sum_{b=\pm 1} \sum_{k=1}^{\lfloor \frac{\sqrt{24m+25}+1}{6} \rfloor} 
     (-1)^k q^{k(3k+b)/2} + 
     q^{m+2} \cdot \poly_{2,m}(f; q)}\right). 
\end{align} 
\end{subequations} 
\end{lemma} 
\begin{proof} 
To justify \eqref{eqn_bmp1_coeff_exp_va}, we observe that for all integers 
$m \geq 1$ and $1 \leq i \leq m$, we have that 
\begin{equation*} 
[q^i] \left(L_f(q) - \sum_{n > m} \frac{f(n) q^n}{1-q^n}\right) = 0, 
\end{equation*} 
i.e., that the $m^{th}$ partial sums of $L_f(q)$ accurately generate $f(k)$ for 
$1 \leq k \leq m$, 
which is easy enough to see by considering the numerator multiples, $q^n$, of the 
geometric series, $(1-q^n)^{-1}$, in the individual Lambert series terms from 
\eqref{eqn_LambertSeriesfb_def}. 
The result in \eqref{eqn_bmp1_coeff_exp_vb} follows immediately from 
\eqref{eqn_bmp1_coeff_exp_va} by combining the terms in the first partial sum, and 
implies the third result in \eqref{eqn_bmp1_coeff_exp_vc} in two key ways. 

First, the respective form of the denominator terms in 
\eqref{eqn_bmp1_coeff_exp_vc} follows from the statement of 
\emph{Euler's pentagonal number theorem}, which states that 
\citep[\S 19.9, Thm.\ 353]{HARDYANDWRIGHT} 
\begin{align*} 
\QPochhammer{q}{q}{\infty} & = \sum_{n=-\infty}^{\infty} (-1)^n q^{n(3n+1)/2} = 
     1 + \sum_{n \geq 1} (-1)^n \left(q^{k(3k-1)/2} + q^{k(3k+1)/2}\right) \\ 
     & = 
     1 - q - q^2 + q^5+q^7-q^{12}-q^{15}+q^{22}+q^{26} - \cdots. 
\end{align*} 
In particular, we see that the pentagonal number theorem shows that 
\begin{equation*} 
[q^i] (1-q)(1-q^2)\cdots(1-q^n) = 
     \begin{cases} 
     (-1)^k, & \text{ if $i = \frac{k(3k\pm 1)}{2}$; } \\ 
     0, & \text{otherwise, } 
     \end{cases} 
\end{equation*} 
for all $i \leq n$ by a contradiction argument. 
Since $(1-q^i)$ is a factor of $\QPochhammer{q}{q}{n}$ for all $1 \leq i \leq n$, 
we see that both of the numerator and denominator of 
\eqref{eqn_bmp1_coeff_exp_vb} are polynomials in $q$, 
each with degree greater than $m+1$. 
This implies the correctness of the 
denominator polynomial form stated in \eqref{eqn_bmp1_coeff_exp_vc}. 

Secondly, since the geometric series in $q^i$ is expanded by 
\begin{equation*} 
\frac{1}{1-q^i} = \sum_{s \geq 0} q^{si}, 
\end{equation*} 
for each finite $i \geq 1$, we have by the definition of $a_f(n)$ in 
Definition \ref{def_afn_ani_matrix_seqs} that the first $m+1$ terms of the 
numerator expansion in \eqref{eqn_bmp1_coeff_exp_vc} are correct. 
Thus since the numerator in \eqref{eqn_bmp1_coeff_exp_vc} is polynomial in $q$, 
it is also correct in form. 
\end{proof} 

\begin{proof}[Proof of Theorem \ref{prop_bn_recs}]
We use \eqref{eqn_bmp1_coeff_exp_vc} in Lemma \ref{lemma_LS_partial_sum_exps} 
to prove our result. 
If we let $\Num_m(q)$ and $\Denom_m(q)$ denote the numerator and denominator 
polynomials in \eqref{eqn_bmp1_coeff_exp_vc}, respectively, we see that by 
definition, $\deg_q \left\{\Num_m(q)\right\} < \deg_q \left\{ \Denom_m(q) \right\}$. 
For any sequence, $(f_n)_{n \geq 0}$, generated by a rational generating function 
of the form 
\begin{align*} 
\sum_{n \geq 0} f_n q^n & = \frac{a_0+a_1q+a_2q^2+\cdots+a_{k-1}q^{k-1}}{ 
     1 - b_1 q - b_2 q^2 - \cdots - b_k q^{k}}, 
\end{align*} 
for some fixed finite integer $k \geq 1$, we can prove that $f_n$ satisfies 
at most a  
$k$-order finite difference equation with constant coefficients of the form 
\cite[\S 2.3]{GFLECT} 
\begin{align*} 
f_n & = \sum_{i=1}^{\min(k, n)} b_i f_{n-i} + a_n \Iverson{0 \leq n < k}. 
\end{align*} 
Then since we define $f(n) = 0$ for all $n < 1$ in 
\eqref{eqn_LambertSeriesfb_def}, and 
since the $m^{th}$ partial sums of $L_f(q)$ generate $g_f(i)$ for all 
$1 \leq i \leq m$ by the lemma, and since $g_f(i) = 0$ for all $i < 1$, 
we see that \eqref{eqn_bmp1_coeff_exp_vc} implies our result. 
\end{proof} 

\begin{proof}[Proof of Theorem \ref{prop_MatrixEquations_fi_eq_AinvBbm}] 
The theorem is a consequence of 
Definition \ref{def_afn_ani_matrix_seqs} applied to 
Theorem \ref{prop_bn_recs}. 
Specifically, by rearranging terms in the result from the previous theorem, 
we see that 
\begin{align*} 
\tag{i} 
A_n \begin{bmatrix} f(1) \\ f(2) \\ \vdots \\ f(n) \end{bmatrix} & = 
     \begin{bmatrix} B_{g_f,0} \\ B_{g_f,1} \\ \vdots \\ B_{g_f,n-1} \end{bmatrix}, 
\end{align*} 
where \eqref{eqn_fn_matrix_eqn} defines the sequence of $B_{g_f,m}$. 
Then by the definition of $a_{n,i}$ given in 
Definition \ref{def_afn_ani_matrix_seqs} 
(\cf Remark \ref{remark_authors_note_on_AnAnInvEntryForms}), 
it is easy to see that 
$A_n$ is lower triangular with ones on its diagonals, and so is 
invertible for all $n \geq 1$. 
Thus by applying $A_n^{-1}$ to both sides of (i), we have proved 
\eqref{eqn_fn_matrix_eqn} in the statement of the theorem. 
\end{proof} 

\begin{proof}[Proof of Corollary \ref{cor_Sigmabx_recs}]
We can show directly by computation that the statement is true for $n = 1$. 
For some $j \geq 1$, suppose that the hypothesis in \eqref{eqn_Sigmabx_recs} is 
true for $n = j$. Then wee see that 
\begin{align*} 
\widetilde{\Sigma}_{g_f,j+1} & = 
     \sum_{b = \pm 1}  
     \sum_{k=1}^{\left\lfloor \frac{\sqrt{24j+25}-b}{6} \right\rfloor}
     (-1)^{k} \left[\Sigma_{g_f,j+1-k(3k+b)/2} + g_f(j+2-k(3k+b)/2)\right] \\ 
     & \phantom{=\sum1\ } + 
     \sum_{k=1}^{j+1} a_f(k+1) \\ 
     & = 
     \Sigma_{g_f,j+1} + \sum_{b=\pm 1} 
     \sum_{k=1}^{\left\lfloor \frac{\sqrt{24j+25}-b}{6} \right\rfloor} 
     g_f(j+2-k(3k+b)/2) + a_f(j+2),\text{by hypothesis} \\ 
     & = 
     \Sigma_{g_f,j+1} + g_f(j+2) \\ 
     & = 
     \Sigma_{g_f,j+2}. 
\end{align*} 
The second to last of the previous equations follows from 
Theorem \ref{prop_bn_recs}, the fact that 
$\lfloor (\sqrt{24n+25}-b)/6 \rfloor \geq \lfloor (\sqrt{24n+1}-b)/6 \rfloor$, and 
since $g_f(i) = 0$ for all $i < 1$. 
\end{proof} 

\section{Examples of the new results} 
\label{Section_Examples} 

\subsection{The generalized sum-of-divisors functions} 

For any $n, x \geq 0$, we have the following recurrence relations 
following from the results proved in Theorem \ref{prop_bn_recs} and 
Corollary \ref{cor_Sigmabx_recs}\footnote{ 
     \label{footnote_pentagonal_number_sets} 
     Here, we make use of a natural conjecture from 
     \eqref{eqn_afn_special_cases} that 
     $a_n(m) = (-1)^k (m+1) \Iverson{m = k(3k\pm 1) / 2}$, where the 
     \emph{pentagonal numbers}, $\omega_{k,b} := \frac{k(3k+b)}{2}$, between 
     $1$ and $n+1$ are given by the following sets for each respective $b = \pm 1$ 
     (see footnote \ref{footnote_expl_of_floored_terms} on page 
     \pageref{footnote_expl_of_floored_terms}): 
     \begin{equation*} 
     \left\{\omega_{1,b}, \omega_{2,b}, \ldots, 
          \omega_{\left\lfloor (\sqrt{24n+1}-b) / 6 \right\rfloor,b} 
          \right\}. 
     \end{equation*} 
}: 
\begin{align*} 
\sigma_1(n+1) & = \sum_{b = \pm 1}  
     \sum_{k=1}^{\lfloor \frac{\sqrt{24m+1}-b}{6} \rfloor} 
     (-1)^{k+1} \sigma_1(n+1-k(3k+b)/2) \\ 
     & \phantom{=\sum1\ } + 
     (-1)^k (n+1) \Iverson{n+1 = k(3k\pm 1)/2} \\ 
\Sigma_{n,x+1} & = \sum_{b = \pm 1} \left( 
     \sum_{k=1}^{\lfloor \frac{\sqrt{24x+1}-b}{6} \rfloor + 1} 
     (-1)^{k+1} \Sigma_{n,x+1-k(3k+b)/2} + 
     \sum_{k=1}^{\lfloor \frac{\sqrt{24x+25}-b}{6} \rfloor} 
     (-1)^{k+1} \frac{k(3k+b)}{2} 
     \right). 
\end{align*} 
Notice that the previous two equations imply exact closed-form formulas for 
$\sigma_1(m)$ and $\Sigma_{n,m}$ at each $m \geq 1$, and similarly for 
fixed $m$ and all $1 \leq n \leq m$. 
Moreover, since we conjecture that the zeros of the polynomials, 
$\widetilde{Q}_n(q) := q^n \cdot Q_n(1/q)$, or alternately the reciprocal zeros 
of the polynomials 
\begin{align} 
\label{eqn_PolyQnq_ZerosRemark_formula}
Q_n(q) & := 1 - \sum_{b = \pm 1} \sum_{k=1}^{\lfloor \frac{\sqrt{24n+1}-b}{6} \rfloor} 
     (-1)^k q^{k(3k+b)/2}, 
\end{align} 
have maximum magnitude of a little over $1$ (depending on $n$), 
we remark as another potential application, 
which is suggested for further work based on the results in this article, that it may be 
possible to use these results to obtain better error bounds on the known average order sums 
\citep[\S 27.11]{NISTHB} 
\begin{align*} 
\sum_{n \leq x} \sigma_1(n) & = \frac{\pi^2}{12} x^2 + O(x \log x) \\ 
\sum_{n \leq x} \sigma_{\alpha}(n) & = \frac{\zeta(\alpha+1)}{\alpha + 1} 
     x^{\alpha+1}  + O(x^{\beta}),\ \alpha > 0, \alpha \neq 1, 
     \beta = \max(1, \alpha), 
\end{align*} 
using the explicit sequence formulas in \eqref{eqn_afn_special_cases}. 

\subsection{Euler's totient function} 

We provide a computation of \eqref{eqn_fn_matrix_eqn} to demonstrate the 
utility to our method: 
\begin{align*} 
\begin{bmatrix} \phi(1) \\ \phi(2) \\ \phi(3) \\ \phi(4) \\ \phi(5) \\ 
                \phi(6) \\ \phi(7) \end{bmatrix} & = 
     \begin{bmatrix} 1 & 0 & 0 & 0 & 0 & 0 & 0 \\ 
                     0 & 1 & 0 & 0 & 0 & 0 & 0 \\ 
                     1 & 1 & 1 & 0 & 0 & 0 & 0 \\ 
                     2 & 1 & 1 & 1 & 0 & 0 & 0 \\ 
                     4 & 3 & 2 & 1 & 1 & 0 & 0 \\ 
                     5 & 3 & 2 & 2 & 1 & 1 & 0 \\ 
                     10 & 7 & 5 & 3 & 2 & 1 & 1 
     \end{bmatrix} 
     \begin{bmatrix} 1 \\ 1 \\ 0 \\ -1 \\ -2 \\ -2 \\ -2 \end{bmatrix} = 
     \begin{bmatrix} 1 \\ 1 \\ 2 \\ 2 \\ 4 \\ 2 \\ 6 \end{bmatrix}. 
\end{align*} 
In this case, we can solve for the right-hand-side vector, $(B_{g_f,n})$, 
explicitly. More precisely, when $(f, g_f) := (\phi(n), n)$ from 
\eqref{eqn_WellKnown_LamberSeries_Examples}, 
we see by straightforward summation that 
\begin{align*} 
B_{n,m} & = m+1 - \frac{1}{8}\Biggl(8 - 5 \cdot (-1)^{u_1} - 4 \left( 
     -2 + (-1)^{u_1} + (-1)^{u_2}\right) m + 
     2 (-1)^{u_1} u_1 (3u_1+2) \\ 
     & \phantom{=m+1 - \frac{1}{8}\Biggl(8\ } + 
     (-1)^{u_2} (6u_2^2+8u_2-3)\Biggr),  
\end{align*} 
where $u_1 \equiv u_1(m) := \lfloor (\sqrt{24m+1}+1)/6 \rfloor$ and 
$u_2 \equiv u_2(m) := \lfloor (\sqrt{24m+1}-1)/6 \rfloor$. 
The first terms of the sequence, $\{B_{n,m}\}_{m \geq 0}$, corresponding to the 
Lambert series over Euler's phi function are given by 
\[
\{B_{n,m}\}_{m \geq 0} = \{1, 1, 0, -1, -2, -2, -2, -1,0,1,2,3,3,3,3,2,1,0,-1,-2,-3, \ldots\}. 
\]

\subsection{The M\"obius function} 

We similarly provide a computation of \eqref{eqn_fn_matrix_eqn} to demonstrate the 
utility to our method in this special case: 
\begin{align*} 
\begin{bmatrix} \mu(1) \\ \mu(2) \\ \mu(3) \\ \mu(4) \\ \mu(5) \\ \mu(6) \\ \mu(7) \end{bmatrix} & = 
     \begin{bmatrix} 1 & 0 & 0 & 0 & 0 & 0 & 0 \\ 
                     0 & 1 & 0 & 0 & 0 & 0 & 0 \\ 
                     1 & 1 & 1 & 0 & 0 & 0 & 0 \\ 
                     2 & 1 & 1 & 1 & 0 & 0 & 0 \\ 
                     4 & 3 & 2 & 1 & 1 & 0 & 0 \\ 
                     5 & 3 & 2 & 2 & 1 & 1 & 0 \\ 
                     10 & 7 & 5 & 3 & 2 & 1 & 1 
     \end{bmatrix} 
     \begin{bmatrix} 1 \\ -1 \\ -1 \\ 0 \\ 0 \\ 1 \\ 0 \end{bmatrix} = 
     \begin{bmatrix} 1 \\ -1 \\ -1 \\ 0 \\ -1 \\ 1 \\ -1 \end{bmatrix}. 
\end{align*} 
In this case the vector, $(B_{g_f,m})$, is given by the formula 
\begin{align*} 
B_{\Iverson{n = 1}, m} & = \Iverson{m = 0} + \sum_{b = \pm 1} 
     \sum_{k=1}^{\lfloor \frac{\sqrt{24m+1}-b}{6} \rfloor} 
     (-1)^k \Iverson{m+1-k(3k+b)/2 = 1}. 
\end{align*} 
The first terms of the sequence, $\{B_{g_f,m}\}_{m \geq 0}$, for the 
Lambert series over the M\"obius function are given by 
\[
\left\{B_{\Iverson{n = 1}, m}\right\}_{m \geq 0} = 
     \{1, -1, -1, 0, 0, 1, 0, 0, 1, 0, 0, 0, 0, -1, 0, 0, -1, 0, 0, 0, 0, \ldots\}. 
\]

\subsection{Liouville's lambda function} 

Finally, we provide a computation of \eqref{eqn_fn_matrix_eqn} to demonstrate the 
utility to our method in the case of the Liouville lambda function: 
\begin{align*} 
\begin{bmatrix} \lambda(1) \\ \lambda(2) \\ \lambda(3) \\ \lambda(4) \\ \lambda(5) \\ 
                \lambda(6) \\ \lambda(7) \end{bmatrix} & = 
     \begin{bmatrix} 1 & 0 & 0 & 0 & 0 & 0 & 0 \\ 
                     0 & 1 & 0 & 0 & 0 & 0 & 0 \\ 
                     1 & 1 & 1 & 0 & 0 & 0 & 0 \\ 
                     2 & 1 & 1 & 1 & 0 & 0 & 0 \\ 
                     4 & 3 & 2 & 1 & 1 & 0 & 0 \\ 
                     5 & 3 & 2 & 2 & 1 & 1 & 0 \\ 
                     10 & 7 & 5 & 3 & 2 & 1 & 1 
     \end{bmatrix} 
     \begin{bmatrix} 1 \\ -1 \\ -1 \\ 1 \\ -1 \\ 0 \\ 0 \end{bmatrix} = 
     \begin{bmatrix} 1 \\ -1 \\ -1 \\ 1 \\ -1 \\ 1 \\ -1 \end{bmatrix}. 
\end{align*} 
In this case the vector, $(B_{g_f,m})$, is given by the formula 
\begin{align*} 
B_{\Iverson{\text{$n$ is a positive square}}, m} & = 
     \Iverson{\sqrt{m+1} \in \mathbb{Z}} \\
     & \phantom{=1\ } - 
     \sum_{b = \pm 1} 
     \sum_{k=1}^{\lfloor \frac{\sqrt{24m+1}-b}{6} \rfloor} 
     (-1)^{k+1} \Iverson{\sqrt{m+1-k(3k+b)/2} \in \mathbb{Z}}. 
\end{align*} 
The first few terms of the sequence, $\{B_{g_f,m}\}_{m \geq 0}$, for the 
Lambert series over Liouville's function are given by 
\[
\left\{B_{\Iverson{\text{$n = k^2$}}, m}\right\}_{m \geq 0} = 
     \{1, -1, -1, 1, -1, 0, 0, 1,2,-1,0,0,-1,1,0,0,-1,-1,-1,0, \ldots\}. 
\]

\section{Conclusions} 
\label{Section_Concl} 

\subsection{Summary} 

We have given proofs of several new recurrence relations and matrix equations for the 
sequences, $f(n)$ and $g_f(n)$, implicit to the Lambert series expansions defined in 
\eqref{eqn_LambertSeriesfb_def} where one of $f(n)$ or $g_f(n)$ 
is typically an interesting arithmetic function we wish to study. 
The key ingredients to the proofs of these results are the definitions of the 
matrices, $A_n$, in Definition \ref{def_afn_ani_matrix_seqs}, and 
Euler's pentagonal theorem applied to the 
partial sums of the left-hand-side series for $L_f(q)$ defined by 
\eqref{eqn_LambertSeriesfb_def}. 

The special case examples from \eqref{eqn_WellKnown_LamberSeries_Examples} 
for $\sigma_1(n)$, $\phi(n)$, $\mu(n)$, and $\lambda(n)$ 
given in Section \ref{Section_Examples} are easily extended to form related 
results for Lambert series of other special functions, such as those given for the 
logarithmic derivatives of the \emph{Jacobi theta functions}, 
$\vartheta_i(z, q)$, cited in the reference \citep[\S 20.5(ii)]{NISTHB}. 
There are also well-known Lambert series expansions involving 
von Mangoldt's function, $\Lambda(n)$, $|\mu(n)|$, the number of distinct 
primes dividing $n$, $\omega(n)$, and Jordan's totient function, $J_t(n)$, 
which provide still other applications of our new results. 
To the best of our knowledge these results, and certainly the interpretations of 
their proofs, are new in the literature. 

\subsection{Generalizations} 

We can generalize these results to form analogous matrix equations and 
new recurrence relations for the sequences, $\widetilde{f}(n)$ and 
$\widetilde{g}(n)$, implicit to the expansions of generalized Lambert series 
of the form 
\begin{align*} 
L_{\widetilde{f}}(\alpha, \beta; a, b, c, d; q) & := \sum_{n \geq 1} 
     \frac{\widetilde{f}(n) \alpha^n q^{d(an+b)}}{1-\beta q^{c(an+b)}} = 
     \sum_{m \geq 1} \widetilde{g}(m) q^m, 
\end{align*} 
for some $\alpha, \beta, a, b, c, d \in \mathbb{C}$ with 
$\alpha, \beta, a, c, d \neq 0$ 
such that $\max(|\alpha q^{da}|, |\beta q^{ca}|) < 1$. 
In particular, if we know the forms of the coefficients of the 
power series expansions of the infinite $q$-Pochhammer products, 
$\QPochhammer{\beta q^{cb}}{q^{ca}}{\infty}$, with respect to $q$ 
(which may or may not be obvious depending on the application), then 
we can extend the proof of Lemma \ref{lemma_LS_partial_sum_exps} to form 
new proofs of analogous results for these generalized Lambert series expansions. 

One immediate application of these generalized results 
is a Lambert series generating function for the 
\emph{sums of squares function}, $r_2(n)$, given by 
\citep[\S 17.10, Thm.\ 311]{HARDYANDWRIGHT} \citep[\S 27.13(iv)]{NISTHB} 
\begin{align*} 
\sum_{n \geq 1} \frac{4 \cdot (-1)^{n+1} q^{2n+1}}{1-q^{2n+1}} & = 
     \sum_{m \geq 1} r_2(m) q^m, && 
     (\widetilde{f}, \widetilde{g}) := ((-1)^{n+1}, r_2(n)). 
\end{align*} 
However, we point out that the coefficients of the power series expansion of the 
$q$-Pochhammer symbol, $\QPochhammer{q}{q^2}{\infty}$, in $q$ does not 
appear to have a known closed-form formula, only related $q$-series expansions 
proved in the references. 
Nonetheless, the study of the analogous results to those proved within this article 
corresponding to these generalized cases is an interesting new direction for 
more careful future study. 

\renewcommand{\refname}{References}

\end{document}